\documentclass[12pt]{amsart}
\usepackage{amsthm, amsmath, amssymb, mathrsfs, enumitem, svn, amsfonts, stmaryrd, mathtools, tikz-cd,mathscinet, amscd}
\usepackage{verbatim, comment} %for comment  

\usepackage[all,cmtip,poly]{xy}

\usepackage[a4paper, bindingoffset=0.2in,  left=0.7in,right=0.7in, top=0.7in,bottom=0.7in, footskip=.25in]{geometry}

\def \sfi {\mathrm{Mod}_{\gs}^{\varphi}}
\def \gs {\mathfrak S}
\def \hR {{\widehat{\mathcal R} }}
\def \hM {{\hat \M}}
\def \O {\mathcal O}

\def \t {\mathrm}
 
\newcommand{\wt}{\widetilde}

\newcommand{\M}{\mathfrak{M}}

\newcommand{\cM}{\mathcal{M}}
\newcommand{\cm}{\mathcal{M}}

\newcommand{\Zp}{\mathbb{Z}_p}
\newcommand{\Qp}{\mathbb{Q}_p}

\newcommand{\Z}{\mathbb{Z}}
\newcommand{\Q}{\mathbb{Q}}

\DeclareMathOperator{\Fil}{Fil}

\DeclareMathOperator{\Gal}{Gal}
\DeclareMathOperator{\GL}{GL}

\DeclareMathOperator{\Hom}{Hom}

\DeclareMathOperator{\Mat}{Mat}

\DeclareMathOperator{\Mod}{Mod}

\DeclareMathOperator{\Rep}{Rep}
\DeclareMathOperator{\rank}{rank}

\DeclareMathOperator{\tor}{tor}
\DeclareMathOperator{\fr}{fr}

\newcommand{\cris}{\mathrm{cris}}

\newcommand{\Frob}{\mathrm{Frob}}

\newcommand{\st}{\mathrm{st}}

\newcommand{\MF}{MF^{(\varphi, N)}}

\newcommand{\bigMF}{\mathcal{MF}^{(\varphi, N)}}

\newcommand{\Acris}{A_{\textnormal{cris}}}

\newcommand{\Asthat}{\widehat{A_{\textnormal{st}}}}

\newcommand{\bigM}{\mathcal{M}}
\newcommand{\bigD}{\mathcal{D}}

\newcommand{\D}{\mathcal{D}}
\newcommand{\huaS}{\mathfrak{S}}
\newcommand{\huaM}{\mathfrak{M}}

\newcommand{\huaN}{\mathfrak{N}}

\newtheorem{thm}{Theorem}[section]  
\newtheorem{theorem}[thm]{Theorem}
\newtheorem{prop}[thm]{Proposition}

\newtheorem{lemma}[thm]{Lemma}
\newtheorem{cor}[thm]{Corollary}

\theoremstyle{definition}
\newtheorem{defn}[thm]{Definition}

\newtheorem{remark}[thm]{Remark}

\numberwithin{equation}{subsection}

\begin{document}
\title[]{Fontaine-Laffaille modules and strongly divisible modules}
\date{\today} 
\author[]{HUI GAO}
%\address{Department of Mathematics and Statistics, University of Helsinki, FI-00014, Finland}
%\email{hui.gao@helsinki.fi}
%\subjclass{Primary 14F30,14L05}

\begin{abstract}
In this note, we study the relation between Fontaine-Laffaille modules and strongly divisible modules, \emph{without} assuming the main theorem of Fontaine-Laffaille (but we need to assume the main results concerning strongly divisible modules). This in particular gives a new proof for the main theorem of Fontaine-Laffaille (for $p>2$).

Dans cette note, nous \'{e}tudions la relation entre les modules de Fontaine-Laffaille et les modules fortement divisibles, \emph{sans} supposer le th\'{e}or\`{e}me principal de Fontaine-Laffaille (mais nous supposons les principaux r\'{e}sultats concernant les modules fortement divisibles). Cela donne en particulier une nouvelle preuve pour le th\'{e}or\`{e}me principal de Fontaine-Laffaille (pour $p> 2$).

\end{abstract}

\maketitle
 
\tableofcontents

\section{Introduction}
\subsection{Overview and main results}
This note concerns the relation between some categories in integral $p$-adic Hodge theory. In integral $p$-adic Hodge theory, we use various (semi-)linear objects to study $\Zp$-lattices in semi-stable Galois representations. Several theories have been proposed to classify $\Zp$-lattices in semi-stable representations, normally under certain restrictions (e.g., ramification of base field, Hodge-Tate weights, crystalline representations only, etc.). For example, we have Fontaine and Laffaille's theory \cite{FL82} of strongly divisible $W(k)$-lattices, Breuil's theory of strongly divisible $S$-lattices (see \cite{Bre02}),
the theory of Wach modules developed by Wach, Colmez and Berger (\cite{Wach96, Wach97, Col99, Ber04}), and T. Liu's theory of $(\varphi, \hat G)$-modules (\cite{Liu10}) using Kisin modules \cite{Kis06}.
Among all these theories, let us point out that Fontaine and Laffaille's theory is particularly simple, and it has generated many powerful applications, most notably, in Galois deformation theory and modularity lifting theorems (e.g., \cite{Ram93, BLGGT14}).

In order to better understand these various theories in integral $p$-adic Hodge theory, it is natural to study the relations between them. Let us list some of the studies:
\begin{enumerate}[leftmargin=*]
\item The relation between Wach modules and Fontaine-Laffaille modules was studied in \cite{Wach97} and \cite[\S V.2]{Ber04}. In particular, a ``direct" equivalence between the two categories is established in \cite[Thm. 3]{Wach97}.

\item The relation between $(\varphi, \hat G)$-modules and strongly divisible $S$-lattices is quite clear, via an explicit functor from Kisin modules to Breuil modules (see e.g., our Theorem \ref{thm kisbre}).

\item The relation between $(\varphi, \hat G)$-modules and Wach modules was studied in \cite[Thm. 1.0.2]{Liucompa}.

\item The relation between Fontaine-Laffaille modules and strongly divisible $S$-lattices (and $(\varphi, \hat G)$-modules) was studied in \cite[Prop. 4.1.2(6)]{Liucompa}. Note that in order to show the compatibility result as in \cite[Prop. 4.1.2(6)]{Liucompa}, it is necessary to assume \emph{all} the theories of Fontaine-Laffaille, Breuil, Kisin and Liu.
\end{enumerate}

Concerning item (4) listed above, there is actually a very obvious functor (denoted as $\underline{\cM}_S$ in our \S \ref{subsec FLtoS}) from Fontaine-Laffaille modules to strongly divisible $S$-lattices (already noted, e.g., in \cite[Example 2.2.2 (2)]{Bre02}). It is quite intriguing to see if one can actually ``directly" prove that the functor induces an equivalence of categories (without assuming the main theorem of Fontaine-Laffaille). This has proved to be a non-trivial task. In fact, as we will see in our paper, we will need to use some highly non-trivial results (\cite{GLS14}) on the shape of Kisin modules (and thus, shape of strongly divisible $S$-lattices), in order to establish the desired equivalence of categories. Note that we also need to assume the main theorem about  strongly divisible $S$-lattices (proved in \cite{Bre02, Liu08, Gao13}).

The relation we established thus gives a new proof of Fontaine-Laffaille theory. Clearly, our ``new" proof is in no way easier than the original proof. However, we would like to point out that our proof gives another approach to Fontaine-Laffaille theory using the more recent tools in integral $p$-adic Hodge theory. In particular, it sheds some new light on the intricate structures in Fontaine-Laffaille modules (as well as other integral theories).

Let us (very roughly) state our main result here (see \S \ref{subsec final} for more details).
\begin{thm} \label{thm intro}
Let $p>2$. There exists a functor $M_{\mathrm{FL}}$ (see \S \ref{defnFL}) from the category of strongly divisible $S$-lattices to the category of Fontaine-Laffaille modules, which is quasi-inverse to $\underline{\cM}_S$.
\end{thm}
%We established a ``direct" equivalence of categories between Fontaine-Laffaille modules and strongly divisible modules (under the condition $p>2$). In particular, we obtained a new proof of (some part of) Fontaine-Laffaille theory.

The above theorem establishes a ``direct" equivalence of categories between Fontaine-Laffaille modules and strongly divisible modules. In particular, we can obtain a new proof of (some part of) Fontaine-Laffaille theory (see \S \ref{subsec final}).

\begin{remark}
The $p=2$ case can also be established, pending some work in preparation \cite{Xiyuanwang}. See Remark \ref{rem p=2} for more detail.
\end{remark}

\subsection{Notations}
Here, we only quickly recall some standard notations in $p$-adic Hodge theory. For more details, the readers can see, e.g., \cite{Gao13}.

Let $p$ be a prime, $k$ a perfect field of characteristic $p$, $W(k)$ the ring of Witt vectors, $K_0 = W(k)[\frac{1}{p}]$ the fraction field, $K$ a finite totally ramified extension of $K_0$, $e =e(K/K_0)$ the ramification index and $G=G_K =\Gal(\overline{K}/K)$ the absolute Galois group for a fixed algebraic closure $\overline{K}$.
Fix a uniformizer  $\pi$ in $K$ and the Eisenstein polynomial $E(u) \in W(k)[u]$ of $\pi$. Define $\pi _n \in \overline K$ inductively such that $\pi_0 = \pi$ and $(\pi_{n+1})^p = \pi_n$. Set $K_{\infty} : = \cup _{n = 1} ^{\infty} K(\pi_n)$, and let $G_{\infty}:= \Gal (\overline K / K_{\infty})$.

We use $\Rep_{\Zp}^{\tor}(G_K)$ (resp. $\Rep_{\Zp}^{\fr}(G_K)$) to denote the category finite $p$-power torsion (resp.  $\Zp$-finite free) representations of $G_K$.
We use $\Rep_{\Zp}^{\fr, \st, [0, r]}(G_K)$ (resp. $\Rep_{\Zp}^{\fr, \cris, [0, r]}(G_K)$) to denote the category of finite free $\Zp$-lattices in semi-stable (resp. crystalline) representations of $G_K$ with Hodge-Tate weights in the range $[0, r]$.

Let $\O_{\overline K}$ be the ring of integers of $\overline K$. Let $R:=\varprojlim \limits_{x\to x^p} \O_{\overline K}/ p \O_{\overline K}$, and let $W(R)$ be the ring of Witt vectors. There are natural Frobenius endomorphisms on $R$ and $W(R)$.
Let $A_{\t {cris}}$ and $B_{\st}$ be the usual period rings (see, e.g., \cite[\S 1]{Gao13}).

Let $\mathfrak{S}=W(k)[\![u]\!]$ with the Frobenius endomorphism $\varphi_{\huaS}: \huaS \to \huaS$ which acts on $W(k)$ via arithmetic Frobenius and sends $u$ to $u^p$.
Via the map $u\mapsto [\underline \pi]$, there is an embedding $\gs \hookrightarrow W(R)$ which is compatible with
Frobenius endomorphisms.

We denote  by $S$  the $p$-adic completion of the divided power
envelope of $W(k)[u]$ with respect to the ideal generated by $E(u)$ (see, e.g., \cite[\S 2.2]{Gao13} for more details).
There is a unique map (Frobenius) $\varphi_S: S \to S$ which extends
the Frobenius on $\gs$. We write $N_S$ for the $W(k)$-linear derivation on $S$ such that $N_S(u)= -u$. Let  $\Fil ^j S\subset S $ be the $p$-adic completion of the ideal generated by $\gamma_i (E(u)):= \frac{E(u)^i}{i!}$ with $ i \geq j$.
For $1 \leq j \leq p-1$, we have $\varphi(\Fil^j S) \subset p^jS$, and we denote $\varphi_j = \frac{\varphi}{p^j}: \Fil^j S \to S$.
We denote $c =\frac{\varphi(E(u))}{p}$ which is a unit in $S$.
The embedding $\huaS \to W(R)$ extends to an embedding $S \hookrightarrow A_\cris$ compatible with Frobenius $\varphi$ and filtration.

When $V$ is a semi-stable representation of $G_K$, we let $D_{\textnormal{st}} (V):  = (B_{\st} \otimes_{\Qp} V^{\vee})$ where $V^{\vee}$ is the dual representation of $V$. The Hodge-Tate weights of $V$ are defined to be $i \in \mathbb Z$ such that $\mathrm{gr}^i D_{\st}(V) \neq 0$.  For example, for the cyclotomic character $\varepsilon_p$, its Hodge-Tate weight is $\{ 1\}$. In this paper, we only consider representations with non-negative Hodge-Tate weights.

\subsection{Acknowledgement}
The author thanks Tong Liu for some related discussions. The paper is written when the author is a postdoc in University of Helsinki. The postdoc position is funded by Academy of Finland, through Kari Vilonen. The author thanks Kari Vilonen for constant support.

\section{Objects in integral $p$-adic Hodge theory}
In this section, we recall some definitions in (rational) $p$-adic Hodge theory and integral $p$-adic Hodge theory.

\subsection{Rational $p$-adic Hodge theory} \label{subrational}
In this subsection, let $r$ be a fixed integer in the range $[0, \infty)$.

A filtered $(\varphi,N)$-module $D$ is a finite dimensional $K_0$-vector space equipped with the usual $\varphi, N$ and a decreasing filtration on $(\Fil^{i}D_{K})_{i\in\mathbb{Z}}$ on $D_{K}=D\otimes_{K_0} K$. In this paper, we only consider $D$ such that $\Fil^0 D_K=D_K$ (i.e., those with non-negative filtration jumps). We use $\MF$ to denote this category. Recall as in \cite[Def. 1.0.4]{Gao13}, we define the following:

\begin{defn}
 For $D \in \MF$ with Hodge-Tate weights in $\{ 0, \ldots, r\}$,
\begin{enumerate}
\item  $D$ is called \'{e}tale if $\Fil^r D_K =D_K$, it is called multiplicative if $\Fil^1 D_K =\{0\}$.
\item $D$ is called nilpotent if it does not have nonzero multiplicative submodules, it is called unipotent if it does not have nonzero \'{e}tale quotients.
\end{enumerate}
\end{defn}

Given a semi-stable representation $V$, let $D$ be the corresponding weakly-admissible filtered $(\varphi, N)$-module (via the main theorem of \cite{CF00}). Then it is easy to show that $D$ is multiplicative if and only if $V$ is an unramified representation, $D$ is nilpotent if and only if $V$ contains no nonzero unramified quotient. Also, $D$ is \'{e}tale (resp. unipotent) if and only the Cartier dual $D^{\vee}$ (see \cite[\S 1]{Gao13}) is multiplicative (resp. nilpotent). We call the representation $V$ \'{e}tale (resp. multiplicative, nilpotent, unipotent) if the corresponding $D$ is \'{e}tale (resp. multiplicative, nilpotent, unipotent).
A $\Zp$-representation $T \in \Rep_{\Zp}^{\fr, \st}(G_K)$ is called \'{e}tale (resp. multiplicative, nilpotent, unipotent) if the $\Qp$-representation $T[\frac 1 p]$ is.

Let $S_{K_0}=S[\frac 1 p]$, and extend $\varphi, N$ actions on $S$ to $S_{K_0}$ ($\Qp$-linearly).
Let $\Fil^i S_{K_0}:= \Fil^i S\otimes_{\Zp}\Qp$.
As in \cite{Bre97}, let $\bigMF$ be the category whose objects are finite free $S_{K_0}$-modules $\D$ with:
\begin{itemize}
 \item a $\varphi_{S_{K_0}}$-semi-linear morphism $\varphi_{\D}: \D \to \D$ such that the determinant of $\varphi_{\D}$ is invertible in $S_{K_0}$;
 \item a decreasing filtration $\{\Fil^i\D\}_{i=0}^{\infty}$ of $S_{K_0}$-submodules of $\D$ such that $\Fil^0\D=\D$ and $\Fil^i S_{K_0} \Fil^j \D \subseteq \Fil^{i+j}\D$;
 \item a $K_0$-linear map $N: \D \to \D$ such that $N(fm)=N(f)m+fN(m)$ for all $f\in S_{K_0}$ and $m \in \D$, $N\varphi=p \varphi N$ and $N (\Fil^i \D) \subseteq \Fil^{i-1}\D$.
\end{itemize}
Morphisms in the category are $S_{K_0}$-linear maps preserving filtrations and commuting with $\varphi$ and $N$. A sequence $0 \to \D_1 \to \D \to
\D_2 \to 0$ is called short exact if it is short exact as $S_{K_0}$-modules and the sequences on filtrations $0 \to \Fil^i \D_1 \to \Fil^i \D \to
\Fil^i \D_2 \to 0$ are short exact for all $i$. We call $\D_2$ a quotient of $\D$ in this case.

For $D \in \MF$, we can associate an object in $\bigMF$ by $\bigD:= S\otimes_{W(k)}D$ and
 \begin{itemize}
 \item $\varphi: =\varphi_S \otimes \varphi_D$;
 \item $N:= N\otimes Id + Id\otimes N$;
 \item $\Fil^0\bigD :=\bigD$ and inductively,
 $$\Fil^{i+1}\bigD := \{ x \in \bigD | N(x) \in \Fil^i \bigD \text{ and } f_{\pi}(x) \in \Fil^{i+1}D_K  \},$$
 where $f_{\pi}: \bigD \twoheadrightarrow D_K$ by $s(u)\otimes x \mapsto s(\pi)\otimes x$.
 \end{itemize}

\begin{thm} \cite[\S 6]{Bre97} \label{bigD}
The functor above induces an equivalence between $\MF$ and $\bigMF$.
\end{thm}

\begin{defn} Let $\D \in \bigMF$.
\begin{enumerate}
  \item $\D$ is called \'{e}tale if  $\Fil^r \D =\D$, it is called multiplicative if $\Fil^i \D = \Fil^i S_{K_0} \D$ for some $1 \leq i \leq r$ (equivalently, for all $1 \leq i \leq r$).
  \item  $\D$ is called nilpotent if it does not have nonzero multiplicative submodules, it is called unipotent if it does not have nonzero \'{e}tale quotients.
\end{enumerate}
\end{defn}

\begin{lemma}
$\D \in \bigMF$ is \'{e}tale (resp. multiplicative, nilpotent, unipotent) if and only if the corresponding $D \in \MF$ (via Theorem \ref{bigD}) is so.
\end{lemma}
\begin{proof}
See \cite[Prop. 3.2.2]{Gao13}.
\end{proof}

\subsection{Fontaine-Laffaille modules}

\newcommand{\pFL}{\mathrm{'FL}_{W(k)}}
\newcommand{\pFLfr}{\mathrm{'FL}_{W(k)}^{\mathrm{fr}}}
\newcommand{\pFLtor}{\mathrm{'FL}_{W(k)}^{\mathrm{tor}}}

\newcommand{\FL}{\mathrm{FL}_{W(k)}}
\newcommand{\FLfr}{\mathrm{FL}_{W(k)}^{\mathrm{fr}}}
\newcommand{\FLtor}{\mathrm{FL}_{W(k)}^{\mathrm{tor}}}

In this subsection, assume $K=K_0$ (i.e., $K$ is unramified), and let $r$ be a fixed integer in the range $[0, p-1]$.

\begin{defn}\hfill
\begin{enumerate}
  \item Let $\pFL$ be the category of finitely generated $W(k)$-modules $M$ with
  \begin{itemize}
    \item a decreasing filtration $\{\Fil^i M\}_{i=0}^{\infty}$ such that $\Fil^{i+1}M$ is a direct summand of $\Fil^i M$ for all $i$, and $\Fil^0 M =M, \Fil^{r+1}M=\{0\}$,
    \item  $\Frob_{W(k)}$-semi-linear maps $\varphi_i: \Fil^i M \to M$ such that  $\varphi_i\mid_{\Fil^{i+1}M} = p\varphi_{i+1}$.
  \end{itemize}
  Morphisms in $\pFL$ are $W(k)$-linear homomorphisms compatible with filtration and $\varphi_i$.

\item Let $\FL$ be the subcategory of $\pFL$ where $\sum_{i=0}^{r} \varphi_i(\Fil^iM)=M$.

\item Let $\pFLfr$ (resp. $\pFLtor$) be the subcategory of $\pFL$ where $M$ is finite free (resp. torsion) over $W(k)$.

\item Let $\FLfr$ (resp. $\FLtor$) be the subcategory of $\FL$ where $M$ is finite free (resp. torsion) over $W(k)$.
\end{enumerate}
\end{defn}

A sequence $0 \to M_1 \to M \to M_2 \to 0$ in $\pFL$ is called short exact if it is exact as a sequence of $W(k)$-modules, and the induced sequence $0 \to \Fil^i M_1 \to \Fil^i M \to \Fil^i M_2 \to 0$ is also short exact for all $i$. In this case, we call $M_2$ a quotient of $M$.

\begin{defn}Let $M \in \pFL$.
\begin{enumerate}
  \item $M$ is called \'etale if $\Fil^r M=M$. $M$ is called multiplicative if $\Fil^1 M=\{0\}$.
  \item $M$ is called unipotent if does not have non-zero \'etale quotients. $M$ is called nilpotent if it does not have non-zero multiplicative submodules.
\end{enumerate}
\end{defn}

\begin{lemma} \label{lemFLunip}
For $M \in \FLfr$, take any $W(k)$-basis $e=(e_1, \ldots, e_d)$, and let $F \in \Mat_d(W(k))$ be the matrix of $\varphi:=\varphi_0$ with respect to $e$. Then there exists $V \in \Mat_d(W(k))$ such that $FV=p^rI_d$. $M$ is unipotent if and only if $\Pi_{n=0}^{\infty} \varphi^n(V)=0$.
\end{lemma}
\begin{proof}
The existence of $V$ is clear. By \cite[Lem. 2.3.2]{Gao13}, it is easy to see that $\cap_{n=0}^{\infty} \varphi^n(M)$ with its induced filtration is the maximal multiplicative submodule of $M$. Thus, $M$ is nilpotent if and only if $\Pi_{n=0}^{\infty} \varphi^n(F)=0$. By considering Cartier dual in the category $\FLfr$, it is easy to see that $M$ is unipotent if and only if $\Pi_{n=0}^{\infty} \varphi^n(V)=0$.
\end{proof}

We define a  functor $T_\cris^*$ from the category $\FLtor$ (resp. $\FLfr$) to $\Rep_{\Zp}^{\tor}(G_{K})$ (resp. $\Rep_{\Zp}^{\fr}(G_{K})$ ):
$$T_\cris^* (M): = \Hom_{W(k), \varphi_i, \Fil ^i} (M, \Acris \otimes_{\Z_p} (\Q_p/\Z_p)) \text{ if } M \in \FLtor , $$
 and
$$T_\cris^* (M) := \Hom_{W(k), \varphi_i, \Fil ^i} (M, \Acris) \text{ if } M \in \FLfr.$$

The following is the main theorem of what we nowadays call Fontaine-Laffaille theory.
\begin{thm}\cite{FL82} \label{thm FL}
\begin{enumerate}
  \item When $0 \leq r \leq p-2$
  \begin{enumerate}
    \item $T_\cris^* :  \FLtor \to \Rep_{\Zp}^{\tor}(G_{K})  $ is exact and fully faithful. The essential image is closed under taking sub-objects and quotients.
    \item $T_\cris^* :  \FLfr \to \Rep_{\Zp}^{\fr, \cris, [0, r]}(G_{K})  $ is an equivalence of categories.
  \end{enumerate}

  \item When $0 \leq r \leq p-1$
  \begin{enumerate}
   \item Restricting to subcategories of unipotent objects, $T_\cris^* :  \FLtor \to \Rep_{\Zp}^{\tor}(G_{K})  $ is exact and fully faithful. The essential image is closed under taking sub-objects and quotients.
    \item Restricting to subcategories of unipotent objects, $T_\cris^* :  \FLfr \to \Rep_{\Zp}^{\fr, \cris, [0, r]}(G_{K})  $ is an equivalence of categories.
  \end{enumerate}
\end{enumerate}
\end{thm}

\begin{remark}
\begin{enumerate}
  \item Note that in the original paper \cite{FL82}, they are using a ``primitive version" of $\Acris$, see the explanation in \cite[\S 3.2.1]{Bre98ENS}, or \cite[\S 4]{Hattori-Rennes}.
  \item Theorem \ref{thm FL}(1)(b) and (2)(b) were not written down in \cite{FL82}. Written proofs can be found in \cite{Bre99SMF} and \cite{GL14} respectively.
\end{enumerate}
\end{remark}

\subsection{Kisin modules}
In this subsection, let $r$ be a fixed integer in the range $[0, \infty)$.

Let $\sfi$ be the category of finite free $\gs$-modules $\M$ equipped with
a $\varphi_{\gs}$-semilinear endomorphism $\varphi_\M : \M \to \M$ such that the cokernel of the linearization $1\otimes \varphi: \huaS\otimes_{\varphi, \huaS}\M \to \M$ is killed by $E(u)^r$. Morphisms in $\sfi$ are $\varphi$-compatible $\gs$-module homomorphisms. We call objects in $\sfi$ (finite free) \emph{Kisin modules of $E(u)$-height $r$}.

We call $\huaM_2$ a quotient of $\huaM$ in the category $\sfi$ if there is a short exact sequence $0 \to \huaM_1 \to \huaM \to \huaM_2 \to 0$ in the category, where short exact means short exact as $\huaS$-modules.

For any finite free Kisin module $\M \in \t{Mod}_{\gs} ^{\varphi, r}$, we define
$$T_\gs (\M) : = \Hom_{\gs, \varphi} (\M , W(R)).$$
$T_\gs (\M)$ is a finite free $\Z_p$-representation of $G_\infty$ and $\rank_{\Z_p} T_\gs(\M) = \rank_\gs (\M)$.
We only recall the following result of Kisin (see \cite[\S2.2]{Liu07} for more details on $T_\gs$).

\begin{thm}\cite[Lem. 2.1.15, Prop. 2.1.12]{Kis06} \label{ThuaS}
\begin{enumerate}
\item For any $G_{\infty}$-stable $\Zp$-lattice $T$ in a semi-stable Galois representation $V$ with Hodge-Tate weights in the range $[0, r]$, there always exists an $\huaN \in \Mod_{\huaS}^{\varphi}$ such that $T_{\huaS}(\huaN) \simeq T$.
\item $T_{\huaS}: \Mod_{\huaS}^{\varphi} \to \Rep_{\Zp}(G_{\infty})$ is fully faithful.
\end{enumerate}
\end{thm}
%\begin{proof}
%These are easy consequences of \cite[Lem. 2.1.15, Prop. 2.1.12]{Kis06}.
%\end{proof}

\begin{defn}
Let $\huaM \in \sfi$,
\begin{enumerate}
 \item $\huaM$ is called \'{e}tale (resp. multiplicative) if $1\otimes \varphi(\huaS\otimes_{\varphi, \huaS}\M)$ is equal to $E(u)^{r}\huaM$ (resp. $\huaM$).
\item $\huaM$ is called nilpotent if it has no nonzero multiplicative submodules, it is called unipotent if it has no nonzero \'{e}tale quotients.
\end{enumerate}
\end{defn}

\begin{lemma}
For $\huaM \in \sfi$, pick any basis $\mathfrak e=(\mathfrak e_1, \ldots, \mathfrak e_d)$, and let $A \in \Mat_d(\huaS)$ be the matrix of $\varphi$ with respect to $\mathfrak e$. Then there exists $B \in \Mat_d(\huaS)$ such that $AB=E(u)^rI_d$. $\huaM$ is unipotent if and only if $\Pi_{n=0}^{\infty} \varphi^n(B)=0$.
\end{lemma}
\begin{proof}
See (the proof of) \cite[Prop. 2.1.3]{Gao13}.
\end{proof}

\begin{remark}
Using Kisin modules, \cite{Liu10} developed the theory of $(\varphi, \hat G)$-module; but we do not need it in this paper.
\end{remark}

\begin{comment}
\begin{defn}
Following \cite{liu4}, a finite free (resp. torsion) \emph{$(\varphi, \hat G)$-module of height $ r$} is a triple $(\M , \varphi, \hat G)$ where
\begin{enumerate}
\item $(\M, \varphi_\M)\in {'\sfi}$ is a finite free (resp. torsion) Kisin module of height $ r$;
\item $\hat G$ is a continuous $\hR$-semi-linear $\hat G$-action on $\hat \M: =\hR
\otimes_{\varphi, \gs} \M$;
\item $\hat G$ commutes with $\varphi_{\hM}$ on $\hM$, \emph{i.e.,} for
any $g \in \hat G$, $g \varphi_{\hM} = \varphi_{\hM} g$;
\item regard $\M$ as a $\varphi(\gs)$-submodule in $ \hM $, then $\M
\subset \hM ^{H_{K}}$;
\item $\hat G$ acts on $W(k)$-module $M:= \hM/I_+\hR\hM\simeq \M/u\M$ trivially.
\end{enumerate}
 Morphisms between  $(\varphi, \hat G)$-modules are morphisms of Kisin modules that commute with $\hat G$-action  on $\hM$'s.
\end{defn}

\subsubsection{}
  For a finite free  $(\varphi, \hat G)$-module  $\hM=(\M, \varphi, \hat G)$,
  we can associate a $\Z_p[G_K]$-module:
\begin{equation}\label{hatT}
\hat T (\hM) := \t{Hom}_{\hR, \varphi}(\hR \otimes_{\varphi, \gs} \M, W(R)),
\end{equation}
where $G_K$ acts on $\hat T(\hM)$ via $g (f)(x) = g (f(g^{-1}(x)))$
for any $g \in G_K$ and $f \in \hat T(\hM)$.

\begin{theorem}[\cite{liu4}]\label{thm-review}
$\hat T$ induces an anti-equivalence between the category of finite free
$(\varphi, \hat G)$-modules of height $r$ and the category of
$G_K$-stable $\Z_p$-lattices in semi-stable representations of $G_K$
with Hodge-Tate weights in $\{0, \dots, r\}$.
\end{theorem}
\end{comment}

\subsection{Strongly divisible modules}

Let $'\Mod_{S}^{\varphi}$ be the category whose objects are triples $(\cm, \Fil^r \cm, \varphi_r)$ where
\begin{itemize}
  \item $\bigM$ is a finite free $S$-module,
  \item $\Fil^r \cm \subseteq \cm$ is an $S$-submodule which contains $\Fil^r S \cdot \cm$, $\cM/\Fil^r \cM$ is $p$-torsion free,
  \item $\varphi_r : \Fil^r \cm \to \cm$ is a $\varphi$-semi-linear map such that $\varphi_r(sx)=c^{-r}\varphi_r(s)\varphi_r(E(u)^rx)$ for $s \in \Fil^r S$ and $x \in \cm$ (recall that $c =\frac{\varphi(E(u))}{p}$),
\end{itemize}
The morphisms in the category are $S$-linear maps preserving $\Fil^r$ and commuting with $\varphi_r$.
A sequence $0 \to \cm_1 \to \cm \to \cm_2 \to 0$ in $'\Mod_{S}^{\varphi}$ is called short exact if it is short exact as a sequence of $S$-modules, and
the sequence on filtrations  $0 \to \Fil^r\cm_1 \to \Fil^r\cm \to \Fil^r\cm_2 \to 0$ is also short exact. In this case, we call $\cm_2$ a quotient of $\cm$.

Let $\Mod_{S}^{\varphi}$ be the subcategory of $'\Mod_{S}^{\varphi}$ where
\begin{itemize}
  \item The image $\varphi_r(\Fil^r \cm)$ generates $\cm$ over $S$.
\end{itemize}

\begin{comment}
\newcommand{\ModFI}{\mathrm{ModFI}}
Let $\ModFI_{S}^{\varphi}$ be the full subcategory of  $'\Mod_{S}^{\varphi}$ with $\cM \simeq \oplus_{i \in I} S_{n_i}$ with $I$ a finite set and $\varphi_r(\Fil^r \cM)$ generates $\cM$.

Let $\Mod_{S}^{\varphi}$ be the full subcategory of $'\Mod_{S}^{\varphi}$ such that $\cM$ is a finite free $S$-module, $\varphi_r(\Fil^r \cM)$ generates $\cM$, and $\cM/\Fil^r \cM$ is $p$-torsion free. Note that the last condition implies that $(\cM/p^n\cM, \Fil^r \cM/p^n\Fil^r \cM, \varphi_r) \in \ModFI_{S}^{\varphi}$ for all $n$.

 By \cite{Bre99a}, for $\cM \in \ModFI_{S}^{\varphi}$, let
$$T_{\text{cris}}(\cM) := \Hom_{'\Mod_{S}^{\varphi}} (\cM, A_{\text{cris}}[1/p]/A_{\text{cris}}),$$
it is a finite torsion $\Zp$-representation of $G_{\infty}$.

\end{comment}

For $\cM \in \Mod_{S}^{\varphi}$ which is finite free of rank $d$, let
$$T_{\text{cris}}(\cM) := \Hom_{'\Mod_{S}^{\varphi}} (\cM, A_{\text{cris}}),$$
it is a finite free $\Zp$-representation of $G_{\infty}$ of rank $d$.

\begin{defn} \label{defunip}
For $\cM \in \Mod_{S}^{\varphi}$,
\begin{enumerate}
\item  $\cM$ is called \'{e}tale if $\Fil^r \cM =\cM$, it is called multiplicative if $\Fil^r \cM =\Fil^rS \cM$.
\item $\cM$ is called nilpotent if it has no nonzero multiplicative submodules, it is called unipotent if it has no nonzero \'{e}tale quotients.
\end{enumerate}
\end{defn}

\subsubsection{} \label{phiphir}
By \cite[Rem. 2.2.1]{Gao13}, for $\cM \in \Mod_{S}^{\varphi}$, $\varphi_r$ induces a map $\varphi: \cm \to \cm$ via $\varphi(x):=\frac{\varphi_r(E(u)^r x)}{c^r}$ (so $\varphi_r =\frac{\varphi}{p^r}$). And one can change the triple $(\cM, \Fil^r \cM, \varphi_r)$ in the definition of $\Mod_{S}^{\varphi}$ to a new triple $(\cM, \Fil^r \cM, \varphi)$ where $\varphi: \cM \to \cM$ such that $\varphi(\Fil^r\cM) \subseteq p^r\cM$ and $\varphi(\Fil^r\cM)$ generates $p^r\cM$. That is, $\varphi_r$ and $\varphi$ provide equivalent information. In the following, we will freely use both of them.

\begin{lemma}\label{lemSunip}
Use notations as in \S \ref{phiphir}. Let $\cM \in \Mod_{S}^{\varphi}$, let $\tilde e=(\tilde e_1, \ldots, \tilde e_d)$ be any basis of $\cm$. Let $\hat A \in \Mat_d(S)$ be the matrix of $\varphi$ with respect to $\tilde e$. Then there exists $\hat B \in \Mat_d(S)$ such that $\hat A \hat B =p^r I_d$.
$\cm$ is unipotent if and only if $\Pi_{n=0}^{\infty} \varphi^n(\hat B)=0$.
\end{lemma}
\begin{proof}
It can be easily deduced from \cite[\S 2.3, \S 2.4]{Gao13}, in particular \cite[Prop. 2.4.5]{Gao13}.
\end{proof}

\subsubsection{} \label{huastoS}
We have an injective map of $W(k)$-algebras $\huaS \hookrightarrow S$ by $u \mapsto u$. Let $\varphi: \huaS \to S$ be the map obtained by composing the injection and the Frobenius on $\huaS$.
We define the functor $\cM_{\huaS}$ from $\Mod_{\huaS}^{\varphi}$ to $\Mod_{S}^{\varphi}$ as follows. Let $\huaM \in \Mod_{\huaS}^{\varphi}$, set $\cM = \cM_{\huaS}(\huaM)=S\otimes_{\varphi,\huaS} \huaM$. We have an $S$-linear map $1 \otimes \varphi: S\otimes_{\varphi,\huaS} \huaM \to S\otimes_{\huaS} \huaM$.  Set
$$\Fil^{r}\cM=\{x\in\cM,(1\otimes\varphi)(x)\in \Fil^{r}S\otimes_{\huaS}\huaM \subseteq S\otimes_{\huaS}\huaM\},$$
and define $\varphi_r: \Fil^r \cM \to \cM$ as the composite:
$$\Fil^r \cM \stackrel{1\otimes \varphi}{\longrightarrow} \Fil^r S \otimes_{\huaS} \huaM \stackrel{\varphi_r\otimes 1}{\longrightarrow} S\otimes_{\varphi, \huaS} \huaM =\cM.$$

\begin{thm} \cite[Thm. 2.3.1]{CL09}, \cite[Thm. 2.5.6]{Gao13}  \label{thm kisbre}
\begin{enumerate}
  \item When $0\leq r \leq p-2$, the functor $\cM_{\huaS}$ induces an equivalence between $\Mod_{\huaS}^{\varphi}$ and $\Mod_{S}^{\varphi}$.
  \item When $0\leq r \leq p-1$, the functor $\cM_{\huaS}$ induces an equivalence between the subcategory of unipotent objects in $\Mod_{\huaS}^{\varphi}$ and the subcategory of unipotent objects in $\Mod_{S}^{\varphi}$.
\end{enumerate}
\end{thm}

Let $'\Mod_{S}^{\varphi, N}$ be the category whose objects are 4-tuples $(\cm, \Fil^r\cm, \varphi_r, N)$, where
\begin{enumerate}
\item $(\cm, \Fil^r\cm, \varphi_r) \in  '\Mod_{S}^{\varphi}$.
\item $N: \cm \to \cm$ is a $W(k)$-linear map such that $N(sx)=N(s)x+sN(x)$ for all $s \in S, x \in \cm$, $E(u)N(\Fil^r \cm) \subseteq \Fil^r \cm$ and the following
diagram commutes:
$$\xymatrix{\Fil^r \cm \ar[d]^{E(u)N} \ar[r]^{\varphi_r} & \cm \ar[d]^{cN}\\ \Fil^r \cm \ar[r]^{\varphi_r} & \cm} $$
\end{enumerate}
Morphisms in the category are given by $S$-linear maps preserving $\Fil^r$ and commuting with $\varphi_r$ and $N$.
A sequence in $'\Mod_{S}^{\varphi, N}$ is said to be short exact if it is short exact as a sequence in $ \Mod_{S}^{\varphi}$.

Let $\Mod_{S}^{\varphi, N}$ be the subcategory of  $'\Mod_{S}^{\varphi, N}$ where $\cm \in \Mod_{S}^{\varphi}$.
We call objects in this subcategory the strongly divisible modules (also called strongly divisible lattices, or strongly divisible $S$-lattices). $\cm \in \Mod_{S}^{\varphi, N}$ is called \'{e}tale (resp. multiplicative, nilpotent, unipotent) if it is so as a module in $\Mod_{S}^{\varphi}$.

\begin{comment}
Let $\ModFI_{S}^{\varphi, N}$ be the full subcategory of $'\Mod_{S}^{\varphi, N}$ consisting of objects such that:
\begin{enumerate}
\item as an $S$-module, $\M \simeq \oplus_{i \in I}S_{n_i}$, where $I$ is a finite set;
\item $\varphi_r(\Fil^r \M)$ generates $\M$ over $S$.
\end{enumerate}

Finally, we denote $\Mod_{S}^{\varphi, N}$ the full subcategory of $'\Mod_{S}^{\varphi, N}$ such that $\M$ is a finite free $S$-module and for all $n$,
$(\M_n, \Fil^r \M_n, \varphi_r, N) \in \ModFI_{S}^{\varphi, N}$, here $\M_n =\M/p^n$.

\end{comment}

For $\cM \in \Mod_{S}^{\varphi, N}$ of $S$-rank $d$,
define $$T_{\textnormal{st}}(\cM):= \Hom_{'\Mod_{S}^{\varphi, N}}(\cM, \Asthat)$$ as in Section 2.3.1 of \cite{Bre99a} (in particular, see \textit{loc. cit.} for the definition of $\Asthat$); it is a finite free $\Zp$-representation of $G_K$ of rank $d$.

\begin{comment}
\begin{defn}
\begin{enumerate}
\item $\mathrm{pMod}_{S}^{\varphi, N}$. Where the module is finite free, has filtration, $N$ is Griffiths
    transverse. But $\varphi_r : \Fil^r \cM \to \cM$ has no requirement.

\item $\mathrm{pMod}_{S}^{\varphi, N, \mathrm{cris}}$ requires $N(\cM) \subset u\cM$.

\item $\mathrm{Mod}_{S}^{\varphi, N}$ requires $\varphi_r$ to generate.

\item $\mathrm{Mod}_{S}^{\varphi, N, \mathrm{cris}}$ requires $\varphi_r$ to generate.
\end{enumerate}
\end{defn}

\end{comment}

\begin{thm}  \label{thm strdivlat}
\begin{enumerate}
  \item When $0 \leq r \leq p-2$, $T_{\textnormal{st}}$ induces an anti-equivalence between $\Mod_{S}^{\varphi, N}$ and $\Rep_{\Zp}^{\fr, \st, [0, r]}(G_K)$.
   \item When $0 \leq r \leq p-1$, $T_{\textnormal{st}}$ induces an anti-equivalence between the subcategory of unipotent objects in  $\Mod_{S}^{\varphi, N}$ and the subcategory of unipotent objects in $\Rep_{\Zp}^{\fr, \st, [0, r]}(G_K)$.
\end{enumerate}
\end{thm}
\begin{proof}
This is first proposed as a conjecture in \cite[Conj. 2.2.6]{Bre02}. Some partial results were known by work of Breuil and Caruso (see \cite[\S 1]{Gao13} for some historical account). (1) is fully established in \cite{Liu08}, and (2) is established in \cite{Gao13}.
\end{proof}

\subsubsection{} \label{subsub fil cm}
By the results in \cite[\S 2.2]{Bre02}, given $\cm \in \Mod_{S}^{\varphi, N}$, then $\D:=\cm[\frac 1 p]$ is naturally an object in $\bigMF$ with $\varphi_{\cm}=\varphi_{\D}|_{\cm}$ (see \ref{phiphir} for $\varphi_{\cm}$), $N_{\cm}=N_{\D}|_{\cm}$, and $\Fil^r \cm =\cm \cap \Fil^r \D$. In particular, the filtration $\{\Fil^i \D\}_{i=0}^{\infty}$ induces a filtration $\{\Fil^i \cm\}_{i=0}^{\infty}$.

Let $\Mod_{S}^{\varphi, N, \cris}$ be the subcategory of $\Mod_{S}^{\varphi, N}$ where $N(\cM) \subset u\cM$.

\begin{cor} \label{cor strdivlat}
\begin{enumerate}
  \item When $0 \leq r \leq p-2$, $T_{\textnormal{st}}$ induces an anti-equivalence between $\Mod_{S}^{\varphi, N, \cris}$ and $\Rep_{\Zp}^{\fr, \cris, [0, r]}(G_K)$.
   \item When $0 \leq r \leq p-1$, $T_{\textnormal{st}}$ induces an anti-equivalence between the subcategory of unipotent objects in  $\Mod_{S}^{\varphi, N, \cris}$ and the subcategory of unipotent objects in $\Rep_{\Zp}^{\fr, \cris, [0, r]}(G_K)$.
\end{enumerate}
\end{cor}
\begin{proof}
For a strongly divisible lattice $\cm \in \Mod_{S}^{\varphi, N}$, it corresponds to an integral crystalline representations if and only if the induced monodromy operator $N$ on $(\cm/u\cm)[\frac 1 p]$ is zero, that is, if and only if $N(\cm) \subset u\cm$.
\end{proof}

\section{Equivalence of categories}
In this section, we fix $0 \leq r \leq p-1$, and always assume $K=K_0$.
In this section, we establish our main results. We first define two functors between Fontaine-Laffaille modules and strongly divisible modules, and then we show they are quasi-inverse to each other.

\subsection{From FL modules to $S$-modules} \label{subsec FLtoS}
In this subsection, we define the functor $\underline{\cM}_S: \pFLfr \to \mathrm{'Mod}_{S}^{\varphi, N, \mathrm{cris}}$.

Let $M \in \pFLfr$, then let $\cm: =\underline{\cM}_S(M) = S\otimes_{W(k)} M$.
Define
\begin{itemize}
  \item $\varphi_{\cm}:=\varphi_S \otimes \varphi_M$,
  \item $N_{\cm}:=N_S\otimes 1$,
  \item $\Fil^r \cm:=\sum_{i=0}^r \Fil^{r-i}S \otimes_{W(k)}\Fil^i M$,
  \item $\varphi_r:= \sum_{i=0}^r \varphi_{r-i, S} \otimes \varphi_{i, M}$.
\end{itemize}
It is easy to check that it is a well-defined functor.

\begin{lemma} \label{lemFLtoS}
Suppose $M \in \pFLfr$, then $M \in \FLfr$ if and only if $\underline{\cM}_S(M) \in \mathrm{Mod}_{S}^{\varphi, N, \mathrm{cris}}$.
\end{lemma}
\begin{proof}
Easy.
\end{proof}

\subsubsection{} \label{notfil}
Let $M$ be a finite free $W(k)$-module with a decreasing filtration $\{\Fil^i M\}_{i=0}^{\infty}$ such that $\Fil^{i+1}M$ is a direct summand of $\Fil^i M$ for all $i$, and $\Fil^0 M =M, \Fil^{r+1}M=\{0\}$.
Let $\cm:=S\otimes_{W(k)} M$, and let $N=N_S\otimes 1: \cm \to \cm$.
Let $f_{\pi}$ be the map $S \to \mathcal O_K=W(k)$ where $u \mapsto \pi$. Also use $f_{\pi}$ to denote the map $f_{\pi}\otimes 1: \cm \to M$.
Denote $s: M \to \cM$ the natural map where $s(x)=1\otimes x$. Then we can use composite of $s$ and $f_{\pi}$ to \emph{identify} $M$ and $f_{\pi}(\cm)$. This identification induces a filtration $\Fil^i (f_{\pi}(\cm))$ on $f_{\pi}(\cm)$ (from the filtration $\Fil^i M$).
Now we can define two decreasing filtrations on $\cm$ as follows.
\begin{enumerate}
  \item (The tensor product filtration). For any $n \geq 0$, let $\Fil_{\otimes}^n \cm: = \sum_{i+j=n} \Fil^i S \otimes_{W(k)} \Fil^j M $.
  \item (The canonical induced filtration). Let $\widehat{\Fil^0} \cm=\cm$, and define inductively, for $n \geq 1$,
$$\widehat{\Fil^n} \cm :=\{ x\in \cm | N(x)\in\widehat{\Fil^{n-1}}\cm, f_\pi(x) \in \Fil^{n}(f_{\pi}(\cm)).  \} $$
\end{enumerate}

\begin{lemma} \label{lemfil1}
Use notations as in \S \ref{notfil}. The two filtrations are the same, i.e., $\Fil_{\otimes}^n \cm=\widehat{\Fil^n} \cm, \forall n \geq 0$.
\end{lemma}
\begin{proof}
It is obvious $\Fil_{\otimes}^n \cm \subset \widehat{\Fil^n} \cm, \forall n \geq 0$, so we only need to show that $\Fil_{\otimes}^n \cm \supset \widehat{\Fil^n} \cm$.

Because $\Fil^{k+1}M$ is a direct summand of $\Fil^k M$ for all $k$, there exists $0 \leq r_1 \leq \ldots \leq r_d \leq r$ and a $W(k)$-basis $(e_1, \ldots, e_d)$ of $M$ such that $e_i \in \Fil^{r_i} M-\Fil^{r_i+1}M$. Then we have
$$ \Fil_{\otimes}^n \cm = \{ \sum_{i=1}^d s_{n-r_i} \otimes e_i \text{ where } s_{n-r_i} \in \Fil^{n-r_i}S  \}.$$
(Here we use the convention $\Fil^k S =S$ when $k<0$). Also, any element in $ \Fil_{\otimes}^n \cm$ can be uniquely expressed in the form $ \sum_{i=1}^d s_{n-r_i} \otimes e_i$ as above.

To show $\Fil_{\otimes}^n \cm \supset \widehat{\Fil^n} \cm$, we use induction on $n$. The case $n=0$ is trivial. Suppose it is true for $n$, and now consider the case $n+1$. For any $x \in  \widehat{\Fil^{n+1}} \cm$, since $x \in \widehat{\Fil^{n}} \cm = \Fil_{\otimes}^n \cm$, we can write it as $x= \sum_{i=1}^d s_{n-r_i} \otimes e_i$. So we need to show that we must have $s_{n-r_j} \in \Fil^{n+1-r_j} S, \forall 1\leq j\leq d$.
\begin{itemize}
  \item For $j$ such that $n-r_j <0$, there is nothing to prove, since $\Fil^{n+1-r_j} S = \Fil^{n-r_j} S=S$.
  \item For $j$ such that $n-r_j\geq 1$, note that $N(x)= \sum_{i=1}^d N(s_{n-r_i}) \otimes e_i$ which is contained in $\widehat{\Fil^n} \cm= \Fil_{\otimes}^n \cm$, so we have $N(s_{n-r_i}) \in \Fil^{n-r_i} S, \forall 1\leq i \leq d$, now we can apply Lemma \ref{easylemma} to conclude $s_{n-r_j} \in \Fil^{n+1-r_j} S$.
  \item For $j$ such that $n-r_j =0$ (if such a $j$ exists), use $f_\pi(x) = \sum_{i=1}^d f_\pi(s_{n-r_i})e_i \in \Fil^{n+1}M$ to conclude that $f_\pi(s_{n-r_j})=0$, which means that $s_{n-r_j} \in \Fil^1 S$.
\end{itemize}
\end{proof}

\begin{lemma} \label{easylemma}
Suppose $i \geq 1$. If $s \in \Fil^{i} S$ and $N(s) \in \Fil^{i} S$, then $s \in \Fil^{i+1} S$.
\end{lemma}
\begin{proof}
Write $s = \sum_{n \geq i} a_n\frac{E(u)^n}{n!}$ with each $a_n \in W(k)$ (recall $K=K_0$), and compute explicitly.
\end{proof}

\subsection{From $S$-modules to FL modules}
In this subsection, we define the functor from $\mathrm{Mod}_{S}^{\varphi, N, \mathrm{cris}}$ to $\FLfr$. In order to do so, we need Proposition \ref{prop S section}. In order to prove Proposition \ref{prop S section}, the following theorem from \cite{GLS14} is a key technical input.

\begin{thm}(\cite[Thm. 4.1]{GLS14}) \label{GLS}
Suppose $p>2$. Let $T$ be a $\Zp$-lattice in a crystalline representation of $G_K$ with Hodge-Tate weights $\{0\leq r_1 \leq \ldots \leq r_d \leq p \}$. Let $\huaM$ be the corresponding Kisin module, then there exists a basis $(\mathfrak e)=(\mathfrak e_1, \ldots, \mathfrak e_d)$ of $\huaM$ such that
$$\varphi(\mathfrak e) = (\mathfrak e)X\Lambda Y,$$
with $X, Y \in \GL_d(\huaS)$ and $\Lambda$ is a diagonal matrix with diagonal entries $E(u)^{r_1}, \ldots, E(u)^{r_d}$.
\end{thm}
\begin{remark} \label{rem p=2}
Theorem \ref{GLS} is the only place we need to assume $p>2$. We were notified that in an ongoing work, X. Wang (\cite{Xiyuanwang}) will extend Theorem \ref{GLS} to the case $p=2$. Once it is established, all results in our paper will also work for $p=2$.
\end{remark}

The proof of the following proposition follows a similar strategy as in \cite[Prop2.4.1]{Liufil}.
\begin{prop} \label{prop S section}
Suppose $p>2$.
Let $\cm \in \mathrm{Mod}_{S}^{\varphi, N, \mathrm{cris}}$. Define $M:=\cm/u\cm$ and equip it with the induced $\varphi$-action. Then there exists a unique $\varphi$-equivariant (injective) section $s: M \to \cm$. It induces an isomorphism
$$1\otimes s: S\otimes_{W(k)} M  \simeq \cM.$$
Moreover, $1\otimes s$ is also $N$-equivariant if we equip an $N$-action on $S\otimes_{W(k)} M $ by $N_S\otimes 1$.
\end{prop}
\begin{proof}
Let $T$ be the $\Zp$-lattice corresponding to $\cm$, and use notations in Theorem \ref{GLS}.
Write $\widetilde{\mathfrak e}: =1\otimes_{\varphi} \mathfrak e$ as the basis for $\cm$ (via the construction in \ref{huastoS}). Then $\varphi(\widetilde{\mathfrak e})= \widetilde{\mathfrak e} \varphi(X\Lambda Y)$. Let $\wt{\mathfrak f}: = \wt{\mathfrak e}Y^{-1}$, then $\varphi(\wt{\mathfrak f}) = \wt{\mathfrak f} Y\varphi(X\Lambda).$
Denote $A=Y\varphi(X\Lambda)$.

Let $s$ be a map sending $\wt{\mathfrak f} (\bmod u)$ to $\wt{\mathfrak f} B$ with $B \in \GL_d(S)$. Then $s$ is a section if and only if
$$Bf_0(A)=A\varphi(B),$$
where $f_0: S \to W(k)$ is the $W(k)$-algebra homomorphism such that $f_0(u)=0$.
Firstly, note that the section $s$ is unique (if it exists), because $\varphi$-equivariant section from $D=M[1/p]$ to $\D=\cm [1/p]$ is unique by \cite[Prop. 6.2.1.1]{Bre97}.
Now, we only need to construct some solution $B \in \GL_d(S)$.

%In order to have a section, suppose $\wt{\mathfrak f} (\bmod u)$ maps via $s$ to  $\wt{\mathfrak f} B$. Then we get $f_0(A)B=A\varphi(B)$, where $f_0: S \to W(k)$ is the $W(k)$-algebra homomorphism such that $f_0(u)=0$.
%Similarly as in \cite[Prop2.4.1]{Liufil}, it suffices to show that $B \in \GL_d(S)$.

Write $f_0(A)=A_0$.
For all $n$, let
$$B_n = A\varphi(A)\cdots \varphi^n(A) \varphi^n(A_0^{-1})\cdots \varphi(A_0^{-1})A_0^{-1},$$
which is equal to $B_0 + \sum_{i=0}^{n-1}(B_{i+1}-B_i)$.
It suffices to show that $B_n  \in \GL_d(S)$, and $p$-adically converges to some matrix in $\GL_d(S)$.
We have
$$B_0 =AA_0^{-1} = Y\varphi(X)\varphi(\Lambda) f_0(\varphi(\Lambda^{-1})) f_0(\varphi(X^{-1})) f_0(Y^{-1}) .$$
The \emph{key thing} here is that $\varphi(\Lambda) f_0(\varphi(\Lambda^{-1})) \in \GL_d(S)$ because $\frac{\varphi(E(u))}{f_0(\varphi(E(u)))} \in S^{\times}$! So $B_0 \in \GL_d(S)$.
It now suffices to show that $B_{i+1}-B_i$ is in $\Mat_d(pS)$ for all $i \geq 0$ and goes to 0 $p$-adically.

By the final paragraph in the proof of \cite[Prop2.4.1]{Liufil}, we have that $A_0^{-1} \in \Mat_d(\frac{W(k)}{p^r})$ (recall that $[0, r]$ is the range of Hodge-Tate weights), and $B_0 =I_d + \frac{u^p}{p^r} Y$ with $Y \in \Mat_d(\huaS)$.
We claim that:
$$(\ast): B_0 =I_d + \frac{u^p}{p} D \text{ with } D \in \Mat_d(S).$$
We will continue with our proof, and prove the claim $(\ast)$ in the end.

Now
$$B_{i+1}-B_i = A\varphi(A)\cdots \varphi^i(A) \varphi^{i+1} (  \frac{u^{p}}{p} D) \varphi^i(A_0^{-1})\cdots \varphi(A_0^{-1})A_0^{-1}. $$
Since $A_0^{-1} \in \Mat_d(\frac{W(k)}{p^r})$, so $\varphi^i(A_0^{-1})\cdots \varphi(A_0^{-1})A_0^{-1}$ is in $\Mat_d(\frac{W(k)}{p^{r(i+1)}})$.
Consider the $p$-power in $\varphi^{i+1} ( \frac{u^{p}}{p} D)$. It is easy to see that
$$\varphi^{i+1} ( \frac{u^{p}}{p} D)   \in \Mat_d (p^{p+\cdots+p^{i+1}}S)$$
Note that $p+\cdots+p^{i+1}- r(i+1) >0, \forall i\geq 0$ (note that $r\leq p-1$), and goes to $\infty$ as $i \to \infty$. This shows that $B_{i+1}-B_i$ is in $\Mat_d(pS)$ for all $i \geq 0$ and goes to 0 $p$-adically. This finishes the proof of our Proposition.

Finally, let us prove the claim $(\ast)$. Write $E(u)=u+pa$ with $a\in W(k)^{\times}$, then
$$ \frac{\varphi(E(u))}{f_0(\varphi(E(u)))} = 1+\frac{u^p}{p}\varphi(a^{-1}) .$$
So $\varphi(\Lambda) f_0(\varphi(\Lambda^{-1})) =I_d +\frac{u^p}{p} C$ for some $C \in \Mat_d(S)$.
So
\begin{eqnarray*}
B_0 &=& P  (I_d +\frac{u^p}{p} C) f_0(P^{-1}) \text{ where } P= Y\varphi(X) \in \GL_d(\huaS) \\
     &=& P f_0(P^{-1})  + \frac{u^p}{p} P C  f_0(P^{-1})\\
   &=& I_d + uQ_1 + \frac{u^p}{p} Q_2 \text{ for } Q_1 \in \Mat_d(\huaS), Q_2 \in \Mat_d(S). \\
\end{eqnarray*}
But $B_0 =I_d + \frac{u^p}{p^r} Y$ with $Y \in \Mat_d(\huaS)$. This implies that $Q_1 \in \Mat_d(u^{p-1}\huaS)$, and so $B_0=I_d + \frac{u^p}{p} D $ with $D \in \Mat_d(S)$.

\end{proof}

\begin{remark}\label{remfunctorial}
\begin{enumerate}
  \item We chose the bases $\wt{\mathfrak f}$ and $\wt{\mathfrak f} (\bmod u)$ so that the matrices $\varphi(\Lambda)$ and $f_0(\varphi(\Lambda^{-1}))$ can \emph{meet each other}. Note that the existence of the section $s$ certainly does not depend on choice of basis. However, if we choose another basis to work, we probably would not be able to prove that the ensuing iteration process gives us a matrix in $\Mat_d(S)$ (although it should!).
  \item The section $s$ is clearly functorial, in the sense that if we have $f:  \cm_1 \to \cm_2$, which induces $f: M_1 \to M_2$, then we have the following commutative diagram:
$$\xymatrix{ M_1 \ar[d]^{f} \ar[r]^{s} & \cm_1 \ar[d]^{f}\\
M_2 \ar[r]^{s} & \cm_2}$$
\end{enumerate}
\end{remark}

\subsubsection{} \label{defnFL}
Now let us define the functor $M_{\mathrm{FL}}: \mathrm{Mod}_{S}^{\varphi, N, \text{cris}} \to \FLfr$ (when $p>2$).
Given $\cm \in \mathrm{Mod}_{S}^{\varphi, N, \text{cris}}$, let $M=M_{\mathrm{FL}}(\cm):= \cm/u\cm$. Equip $M$ with the induced $\varphi$-action.
Now we define a filtration on $M$. By Prop. \ref{prop S section}, we have an $\varphi$-equivariant isomorphism $1\otimes s: S \otimes_{W(k)} M \simeq \cm$, and so we identify $S \otimes_{W(k)} M$ with $\cm$. The filtration $\{\Fil^i \cm\}_{i=0}^\infty$ on $\cm$ (see \ref{subsub fil cm}) induces a filtration on $S \otimes_{W(k)} M$.
Via the surjective map $f_\pi:S \otimes_{W(k)} M \to M$ where $f_\pi(s\otimes m) =s(\pi)m$, we can define $\Fil^i M: =f_{\pi} (\Fil^i ( S \otimes_{W(k)} M )  ).$ Since the sections $s$ are functorial (Remark \ref{remfunctorial}(2)), the definition of $\Fil^i M$ is also functorial.

\begin{lemma}
With notation as above, $\Fil^{i+1}M$ is a direct summand of $\Fil^i M$ for all $i$, and $\Fil^0 M =M, \Fil^{r+1}M=\{0\}$.
\end{lemma}
\begin{proof}
This is an easy consequence of \cite[Prop. 4.5]{GLS14}. Note that $f_{\pi} (\Fil^i (\cm)  )= f_{\pi} (\Fil^i (\huaM^{\ast})  ) $, where $\huaM^{\ast}=\huaS\otimes_{\varphi, \huaS}\M$ as in the notation of \emph{loc. cit.}.
\end{proof}
\begin{remark}
What the above lemma says is that an ``adapted basis" for $\Fil^i \cm$ gives an ``adapted basis" for $\Fil^i M$. The notion of ``adapted basis" is a key idea in the proof of Theorem \ref{GLS}.
\end{remark}

With $\Fil^i M$ in \ref{defnFL}, we can define two filtrations $\Fil_{\otimes}^n \cm$ and $\widehat{\Fil^n} \cm$ on $\cm$ as in \ref{notfil}.

\begin{lemma} \label{fil3}
With notations above, $\Fil_{\otimes}^n \cm = \widehat{\Fil^n} \cm =\Fil^n \cm, \forall n$.
\end{lemma}
\begin{proof}
$\Fil_{\otimes}^n \cm = \widehat{\Fil^n} \cm$ is already proved in Lemma \ref{lemfil1}. The equality $\widehat{\Fil^n} \cm =\Fil^n \cm$ can be easily deduced from the facts recalled in \ref{subsub fil cm}, as we briefly sketch in the following.

Let $\bigD=\cm[\frac 1 p]$, and $D \in \MF$ the corresponding module. By \cite[Prop. 6.2.1.1]{Bre97}, $D \simeq \bigD/u\bigD$ with the isomorphism $(\varphi, N)$-equivariant, and there exists a unique $\varphi$-equivariant section $D \to \bigD$. In our situation, it is precisely the $\Qp$-linear extension of our $s: M \to \cm$ in Prop. \ref{prop S section}. The filtration $\Fil^i \bigD$ is the same as the ``canonical filtration" induced from that of $\Fil^i D_K$ as we recalled in \S \ref{subrational}. And the filtration $\Fil^i D_K$ comes from $f_\pi\circ (1\otimes s)^{-1} : \bigD \simeq S\otimes D \to D_K$, where $f_\pi: S\otimes D \to D_K$ is the map sending $s(u)\otimes d$ to $s(\pi)\otimes d$. This shows that in our situation, we precisely have $\widehat{\Fil^n} \cm =\cm \cap \Fil^n \D =\Fil^n \cm$.
\end{proof}

\begin{prop}
With the above notations, $M=M_{\mathrm{FL}}(\cm)$ is indeed an object in $\FLfr$.
\end{prop}
\begin{proof}
With the identification $1\otimes s: S \otimes_{W(k)} M \simeq \cm$, and the identification $\Fil_{\otimes}^r \cm=\Fil^r \cm$, we can simply apply Lemma \ref{lemFLtoS}.
\end{proof}
This finishes the proof that $M_{\mathrm{FL}}: \mathrm{Mod}_{S}^{\varphi, N, \text{cris}} \to \FLfr$ is a well-defined functor.

\subsection{Equivalence between FL modules and $S$-modules} \label{subsec final}
In this subsection, we show that the functors in the previous two subsections are quasi-inverse to each other, and establish our main results.

\begin{lemma}
The functor $\underline{\cM}_S$ (resp. $M_{\mathrm{FL}}$ when $p>2$) sends unipotent objects to unipotent objects.
\end{lemma}
\begin{proof}
This is easy consequence of the criteria in Lemma \ref{lemFLunip} and Lemma \ref{lemSunip}.
\end{proof}

\begin{theorem} \label{thm equiv fl s}
Suppose $p>2$, $K=K_0$.
\begin{enumerate}
  \item When $0 \leq r \leq p-2$, the functor $\underline{\cM}_S$ induces an equivalence between $\FLfr$ and $\mathrm{Mod}_{S}^{\varphi, N, \text{cris}}$.
  \item When $0 \leq r \leq p-1$, the functor $\underline{\cM}_S$ induces an equivalence between the subcategory of unipotent objects in $\FLfr$ and the subcategory of unipotent objects in $\mathrm{Mod}_{S}^{\varphi, N, \text{cris}}$.
\end{enumerate}
\end{theorem}
\begin{proof}
It suffices to show that $M_{\mathrm{FL}}$ is a quasi-inverse functor in both cases.

Starting from $M \in \FLfr$ (either $0 \leq r \leq p-2$, or $0 \leq r \leq p-1$ with $M$ unipotent), then $\cm:=\underline{\cM}_S(M) = S\otimes_{W(k)}M$. It is easy to see that $M$ and $M_{\mathrm{FL}}(\cm)$ are isomorphic as $W(k)$-modules, with the isomorphism compatible with $\varphi$. The filtration structures are also the same, because
\begin{eqnarray*}
 \Fil^i M_{\mathrm{FL}}(\cm) &=& f_\pi(\Fil^i \cm), \text{ by definition } \\
 &=& f_\pi(\Fil_{\otimes}^i \cm), \text{ by Lemma \ref{fil3} } \\
 &=& \Fil^i M \text{ (obvious)}.
\end{eqnarray*}

Starting from $\cm \in \mathrm{Mod}_{S}^{\varphi, N, \text{cris}}$ (either $0 \leq r \leq p-2$, or $0 \leq r \leq p-1$ with $\cM$ unipotent). Let $M: =M_{\mathrm{FL}}(\cm)$. Then $\cm$ and $\underline{\cM}_S(M)$ are $(\varphi, N)$-equivariantly isomorphic via $1\otimes s$ in Proposition \ref{prop S section}. For the filtration structures, we have
\begin{eqnarray*}
\Fil^i \underline{\cM}_S(M)  &=& \Fil^i_{\otimes} \underline{\cM}_S(M), \text{ by definition } \\
  &=&  \Fil^i \cm ,\text{ by Lemma \ref{fil3}. }
\end{eqnarray*}
%&=&  \widehat{\Fil}^i \underline{\cM}_S(M), \text{ by Lemma \ref{lemfil1} }\\
\end{proof}

As a corollary of Theorem \ref{thm equiv fl s}, we obtain a new proof for Theorem \ref{thm FL}(1)(b), (2)(b).

%\begin{corollary}
%When $p>2$, we obtained a new proof for Theorem \ref{thm FL}(1)(b), (2)(b).
%\end{corollary}
\begin{proof}[Reproof of Theorem \ref{thm FL}(1)(b), (2)(b)]
By Theorem \ref{thm equiv fl s} and Corollary \ref{cor strdivlat}, we only need to check that given $M \in \FLfr$, we have
$$ T_{\cris}^\ast(M) = T_{\st}(\underline{\cM}_S(M)).$$
This is a well known fact, and a written proof can be found in \cite[Prop. 3.2.2]{IM15}.
\end{proof}

\begin{remark}
\begin{enumerate}
  \item Comparing with the original proof in \cite{FL82} (we also recommend the nice exposition in \cite{Hattori-Rennes}), one difference in our ``new" proof of Fontaine-Laffaille theory (that is, using Theorem \ref{thm equiv fl s} and Corollary \ref{cor strdivlat}) avoids explicit classification of simple objects (as in \cite[\S 4]{FL82}) and their related explicit calculations. To be more precise, in the work \cite{Bre02, Kis06, Liu08, Gao13}, explicit classification of simple objects were never needed.

  \item     Actually, we can also obtain a ``new" proof for Theorem \ref{thm FL}(1)(a), by using work in \cite{Carusocrelle} (in particular, \cite[Thm. 1.0.4]{Carusocrelle}). We should also be able to ``reprove" Theorem \ref{thm FL}(2)(a), by extending the work of \cite{Carusocrelle} to the $er=p-1$ unipotent case (indeed, just the $e=1, r=p-1$ unipotent case). To get this ``new" proof for Theorem \ref{thm FL}(1)(a), we will then really need to explicitly classify simple objects, which is done in \cite{Carusocrelle}.
      Let us point out that Theorem \ref{thm FL}(1)(a) and Theorem \ref{thm FL}(2)(a) are very useful results. However, their reproof would deviate from the main scope of our current paper, and we do not see anything more interesting in it, so we choose not to carry it out.
\end{enumerate}
\end{remark}

\bibliographystyle{alpha}

\end{document}